\documentclass[12pt]{article}

\usepackage[colorlinks=true, urlcolor=blue, 
            citecolor=blue, linkcolor=blue]{hyperref}
\usepackage{amsfonts, amsmath, amssymb, amsthm, bm, 
            latexsym, mathrsfs, thmtools, wasysym}
\usepackage{cleveref}
\declaretheorem[name=Theorem, refname={theorem, theorems}, Refname={Theorem, Theorems}, parent=section]{theorem}
\declaretheorem[name=Definition, refname={definition, definitions}, Refname={Definition, Definitions}, sibling=theorem]{definition}
\declaretheorem[name=Lemma, refname={lemma, lemmas}, Refname={Lemma, Lemmas}, sibling=theorem]{lemma}

\usepackage[authoryear]{natbib}
\usepackage{epsfig, graphics, graphicx}
\usepackage{stackengine}
\usepackage{adjustbox, caption, diagbox, fancybox, subcaption}
\usepackage[cmtip, arrow]{xy}
\usepackage{pb-diagram, pb-xy, pgf, tikz, tikz-3dplot}
\usetikzlibrary{arrows}
\usepackage{verbatim} 
\usetikzlibrary{arrows.meta}

\makeatletter
\newcommand\incircbin
{%
  \mathpalette\@incircbin
}
\newcommand\@incircbin[2]
{%
  \mathbin%
  {%
    \ooalign{\hidewidth$#1#2$\hidewidth\crcr$#1\bigcirc$}%
  }%
}

\makeatother

\title{Logically-consistent hypothesis testing 
   \linebreak and the hexagon of oppositions}
\author{Julio Michael Stern \\
University of S\~ao Paulo \\
\and Rafael Izbicki \\
Federal University of S\~ao Carlos \\
\and Luis Gustavo Esteves \\
University of S\~ao Paulo \\
\and Rafael Bassi Stern  \\
Federal University of S\~ao Carlos \\
}
\date{\today}

\makeatletter
\newcommand{\superimpose}[2]{%
  {\ooalign{$#1\@firstoftwo#2$\cr\hfil$#1\@secondoftwo#2$\hfil\cr}}}
\makeatother

\definecolor{green}{rgb}{0.,1.,0.}
\definecolor{red}{rgb}{1.,0.,0.}
\definecolor{qqqqff}{rgb}{0.,0.,1.}
\definecolor{uuuuuu}{rgb}{0.26666666666666666,0.26666666666666666,0.26666666666666666}
\definecolor{zzttqq}{rgb}{0.1,0.2,0.}

\def\argmin \mathop{\rm argmin}
\def\bA{$\widetilde{\mbox{A}}$} 
\def\bE{$\widetilde{\mbox{E}}$} 
\def\bI{$\widetilde{\mbox{I}}$} 
\def\bO{$\widetilde{\mbox{O}}$} 
\def\bU{$\widetilde{\mbox{U}}$} 
\def\bY{$\widetilde{\mbox{Y}}$}
\def\bH{\widetilde{H}}

\def\sqp{\boxplus}
 
\def\varex{\mbox{\Large \varhexagon}}

\def\P{{\mathbb P}}

\def\P{\textrm{I\kern-0.25emP}}

\def\Re{I\kern -0.37 em R}
\def\Na{I\kern -0.37 em N}
\def\Qe{I\kern -0.37 em Q}
\def\Chi2{\mbox{Q}}

\def\g{\,\vert\,}

\def\sc{\widetilde{s}}

\def\g{\,\vert \,}

\def\ev{\,\mbox{ev}\,}

\def\nbp{\mathrel{%
    \mathchoice{\NBP}{\NBP}{\scriptsize\NBP}{\tiny\NBP}%
}}
\def\NBP{{%
    \setbox0\hbox{$\nabla$}%
    \rlap{\hbox to \wd0{\hss\raisebox{1.4pt}{$+$}\hss}}\box0 
}} 
\def\dlp{\mathrel{%
    \mathchoice{\DLP}{\DLP}{\scriptsize\DLP}{\tiny\DLP}%
}}
\def\DLP{{%
    \setbox0\hbox{$\Delta$}%
    \rlap{\hbox to \wd0{\hss\raisebox{0.6pt}{$+$}\hss}}\box0
}} 

\def\nbm{\mathrel{%
    \mathchoice{\NBM}{\NBM}{\scriptsize\NBM}{\tiny\NBM}%
}}
\def\NBM{{%
    \setbox0\hbox{$\nabla$}%
    \rlap{\hbox to \wd0{\hss\raisebox{2.1pt}{$\shortmid$}\hss}}\box0 
}} 
\def\dlm{\mathrel{%
    \mathchoice{\DLM}{\DLM}{\scriptsize\DLM}{\tiny\DLM}%
}}

\def\DLM{{%
    \setbox0\hbox{$\Delta$}%
    \rlap{\hbox to \wd0{\hss\raisebox{1.2pt}{$\shortmid$}\hss}}\box0
}} 

\def\dip{\mathrel{%
    \mathchoice{\DIP}{\DIP}{\scriptsize\DIP}{\tiny\DIP}}\hspace{-1.3mm}}
\def\DIP{{%
    \setbox0\hbox{$\Diamond$}%
    \rlap{\hbox to \wd0{\hss\raisebox{1.1pt}{$+$}\hss}}\box0
}} 

\def\nbb{\nbp} 

\def\dlb{\dlp} 

\def\dib{\dip} 

\def\sqb{\sqp} 

\newcommand{\hexagono}{
 \begin{align*}
  \dgARROWPARTS=12 
  \begin{diagram}  
  \node[2]{\Ovalbox{U: $\Delta H = \square H  \vee \lnot \Diamond H$}} 
	 \arrow{sssw,t,2,-,..}{} 
	 \arrow[4]{s,r,1,=}{}  
	 \arrow{ssse,t,2,-,..}{} \\ 
 \node{\Ovalbox{A: $\square H$}} 
   \arrow{ne,t}{} 
   \arrow[2]{e,b,5,-,--}{}  
	 \arrow[2]{s,l}{} 
   \arrow{ssse,b,2,-,--}{}  
	 \arrow[2]{se,t,4,=}{}    
 \node[2]{\Ovalbox{E: $\lnot \Diamond H$}}  
    \arrow{nw,t}{}  
    \arrow[2]{sw,t,4,=}{}  
		\arrow{sssw,b,2,-,--}{} 
		\arrow[2]{s,r}{} \\  
 \node[2]{\mbox{\varex}} \\ 
 \node{\Ovalbox{I: $\Diamond H $}} 
   \arrow[2]{e,t,5,-,..}{}   
 \node[2]{\Ovalbox{O: $\lnot \square H$}}   \\ 
 \node[2]{\Ovalbox{Y: $\nabla H = \Diamond H \wedge \lnot \square H$}} 
  \arrow{nw,b}{} 
	\arrow{ne,b}{}  
 \end{diagram} 
 \end{align*}
}

\newcommand{\hexagonoEncaixado}{
 \begin{tikzpicture}
   \newdimen\R
   \R=3.3cm

   \foreach \x/\l/\p in
     { 90/{$\Delta H$}/above,
      150/{$\Square H$}/left,
      210/{$ \Diamond H$}/left,
      270/{$ \nabla H$}/below,
      330/{$\lnot \Square H$}/right,
      390/{$\lnot \Diamond H$}/right
     }
     \node[inner sep=0pt,circle,draw,fill,label={\p:\l}] at (\x:\R) {};
     
      \newdimen\Rb
     \Rb=1.65cm
    
      \node at (0.8, 0.6)   (a) {{\small$\lnot \dib H$}};
    
     \node at (0.8, -0.6)   (a) {{\small$\lnot \sqb  H$}};
     
     \node at (-0.9, -0.6)   (a) {{\small$\dib H$}};
     \node at (-0.9, 0.6)   (a) {{\small$\sqb H$}};
     
     \node at (0, 1.2)   (a) {{\small$\dlb H$}};
     
     \node at (0, -1.22)   (a) {{\small$\nbb H$}};

    \newdimen\Rbeps
     \Rbeps=1.65cm
     \newdimen\Reps
     \Reps=3.2cm
      \draw[color=blue] [-{Latex[length=3mm, width=2mm]}]  (30:\Reps)-- (30:\Rbeps);
      \draw[color=blue] [-{Latex[length=3mm, width=2mm]}]  (90:\Reps)-- (90:\Rbeps);
      \draw[color=blue] [-{Latex[length=3mm, width=2mm]}]  (150:\Reps)-- (150:\Rbeps);
      \draw[color=blue] [-{Latex[length=3mm, width=2mm]}]  (210:\Rbeps)-- (210:\Reps);
      \draw[color=blue] [-{Latex[length=3mm, width=2mm]}]  (270:\Rbeps)-- (270:\Reps);
      \draw[color=blue] [-{Latex[length=3mm, width=2mm]}]  (330:\Rbeps)-- (330:\Reps);
      \draw[color=blue] [-{Latex[length=3mm, width=2mm]}]  (390:\R)-- (330:\R);
      \draw[color=blue] [-{Latex[length=3mm, width=2mm]}]  (270:\R)-- (330:\R);
      \draw[color=blue] [-{Latex[length=3mm, width=2mm]}]  (270:\R)-- (210:\R);
      \draw[color=blue] [-{Latex[length=3mm, width=2mm]}]  (150:\R)-- (210:\R);
      \draw[color=blue] [-{Latex[length=3mm, width=2mm]}]  (150:\R)-- (90:\R);
      \draw[color=blue] [-{Latex[length=3mm, width=2mm]}]  (390:\R)-- (90:\R);

      \draw[color=blue] [-{Latex[length=3mm, width=2mm]}]  (390:\Rb)-- (330:\Rb);
      \draw[color=blue] [-{Latex[length=3mm, width=2mm]}]  (270:\Rb)-- (330:\Rb);
      \draw[color=blue] [-{Latex[length=3mm, width=2mm]}]  (270:\Rb)-- (210:\Rb);
      \draw[color=blue] [-{Latex[length=3mm, width=2mm]}]  (150:\Rb)-- (210:\Rb);
      \draw[color=blue] [-{Latex[length=3mm, width=2mm]}]  (150:\Rb)-- (90:\Rb);
      \draw[color=blue] [-{Latex[length=3mm, width=2mm]}]  (390:\Rb)-- (90:\Rb);
\end{tikzpicture}
}

\newcommand{\hexagonoh}{
 \dgARROWPARTS=12
 \begin{diagram}
  \node[2]{\Ovalbox{U: $\dlm H = \square H \vee \lnot \dib H$}} 
  \arrow{sssw,t,2,-,..}{} 
  \arrow[4]{s,r,1,=}{}
  \arrow{ssse,t,2,-,..}{} \\ 
  \node{\Ovalbox{A: $\square H$}}
  \arrow{ne,t}{}
  \arrow[2]{e,b,5,-,--}{}
  \arrow[2]{s,l}{}
  \arrow{ssse,b,2,-,--}{}
  \arrow[2]{se,t,4,=}{}  
  \node[2]{\Ovalbox{E: $\lnot \dib H$}}  
  \arrow{nw,t}{}  
  \arrow[2]{sw,t,4,=}{}  
  \arrow[2]{s,r}{} 
  \arrow{sssw,b,2,-,--}{} \\  
  \node[2]{\mbox{\varex}} \\  
  \node{\Ovalbox{I: $\dib H$}} 
  \arrow[2]{e,t,5,-,..}{}   
  \node[2]{\Ovalbox{O: $\lnot \square H$}}   \\ 
  \node[2]{\Ovalbox{Y: $\nbm H = \dib H \wedge \lnot \square H$}} 
  \arrow{nw,b}{} 
  \arrow{ne,b}{}  
 \end{diagram} 
}

\newcommand{\unionconsonance}{
 \begin{tikzpicture}[line cap=round,
 line join=round,>=triangle 45,x=1.0cm,y=1.0cm]
	\fill[color=zzttqq,fill=zzttqq,fill opacity=0.1] (0.,0.) -- (-3.,-3.) -- (3.,-3.) -- cycle;
	\draw[dotted]   (0.,0.)-- (-3.,-3.);
	\draw[dotted]   (-3.,-3.)-- (3.,-3.);
	\draw[dotted]   (3.,-3.)-- (0.,0.);
	\draw[dotted]   (0.,-1.7)-- (1.0,-1.0);
	\draw[dotted]   (0.,-1.7)-- (-1.0,-1.0);
  \draw [color=qqqqff] [line width=0.25mm] [->]  (0.,-1.7)-- (0.,-3.);
	\draw [color=qqqqff]  [line width=0.25mm] [->] (0.,-3.)-- (-3.,-3.);
	\begin{scriptsize}
	 \draw[color=uuuuuu] (0.1973,0.36586) node {};
	 \draw (1,-2.05) node {$\Diamond (A \cup B \cup C)  $};
	 \draw (0,-3.3) node {$\Diamond (A \cup B )  $};
	 \draw (1.7,-1) node {$\Diamond (B \cup C) $};
	 \draw (-1.7,-1) node {$\Diamond (A \cup C) $};
   \draw (0,0.3) node {$\Diamond C $};
	 \draw[color=zzttqq] (-1.1337,2.0163) node {};
	 \draw[color=zzttqq] (0.0642,2.78828) node {};
	 \draw[color=zzttqq] (1.2621,2.0163) node {};
	 \draw  (-3.3,-3) node {$\Diamond A$};
	 \draw  (3.3,-3) node {$\Diamond B$};
	 \draw[color=zzttqq] (-1.74596,-0.96514) node {};
	 \draw[color=zzttqq] (0.11744,-3.12136) node {};
	 \draw[color=zzttqq] (1.87436,-0.96514) node {};
	\end{scriptsize}
 \end{tikzpicture}
}

\newcommand{\interconsonance}{
 \begin{tikzpicture}[line cap=round,line join=round,>=triangle 45,x=1.0cm,y=1.0cm]
	\fill[color=zzttqq,fill=zzttqq,fill opacity=0.1] (0.,0.) -- (-3.,-3.) -- (3.,-3.) -- cycle;
	\draw[dotted]   (0.,0.)-- (-3.,-3.);
	\draw[dotted]   (-3.,-3.)-- (3.,-3.);
	\draw[dotted]   (3.,-3.)-- (0.,0.);
	\draw[dotted]   (0.,-1.7)-- (1.0,-1.0);
	\draw[dotted]   (0.,-1.7)-- (-1.0,-1.0);
  \draw [color=qqqqff] [line width=0.25mm] [->]  (0.,-1.7)-- (0.,-3.);
	\draw [color=qqqqff]  [line width=0.25mm] [->] (0.,-3.)-- (-3.,-3.);
	\begin{scriptsize}
	 \draw[color=uuuuuu] (0.1973,0.36586) node {};
	 \draw (1,-2.05) node {$\lnot \square (A \cap B \cap C)  $};
	 \draw (0,-3.3) node {$\lnot \square (A \cap B )  $};
	 \draw (1.7,-1) node {$\lnot \square (B \cap C) $};
	 \draw (-1.7,-1) node {$\lnot \square (A \cap C) $};
   \draw (0,0.3) node {$\lnot \square C $};
	 \draw[color=zzttqq] (-1.1337,2.0163) node {};
	 \draw[color=zzttqq] (0.0642,2.78828) node {};
	 \draw[color=zzttqq] (1.2621,2.0163) node {};
	 \draw  (-3.3,-3) node {$\lnot \square A$};
	 \draw  (3.3,-3) node {$\lnot \square B$};
	 \draw[color=zzttqq] (-1.74596,-0.96514) node {};
	 \draw[color=zzttqq] (0.11744,-3.12136) node {};
	 \draw[color=zzttqq] (1.87436,-0.96514) node {};
	\end{scriptsize}
 \end{tikzpicture}
}

\newcommand{\invertibility}{
 \tdplotsetmaincoords{70}{0}
 \begin{tikzpicture}[tdplot_main_coords,scale=1.1, every node/.style={scale=1.0}]
  \def\RI{2}
  \def\RIFora{2.7}
  \def\RII{2}
  \def\RIIFora{2.7}
  \def\RIII{2}
  \def\RIIII{3}
  \def\aa{20}
  \def\ab{80}
  \def\ac{140}
  \def\ad{200}
  \def\ae{260}
  \def\af{320}
  \def\ag{380}

  \draw[thick] (\aa:\RI)
   \foreach \x in {\aa,\af,\ae,\ad} { --  (\x:\RI) node at (\x:\RI) (R1-\x) {} };
  \draw[dashed,thick] (R1-\aa.center)
   \foreach \x in {\ab,\ac,\ad} { --  (\x:\RI) node at (\x:\RI) (R1-\x) {} };
  \path[fill=gray!30] (\RI,0)
   \foreach \x in {\aa,\ab,\ac,\ad,\ae,\af,\aa} { --  (\x:\RI)};

  \begin{scope}[yshift=2.8cm]
   \draw[thick,fill=gray!30,opacity=0.2] (\aa:\RII)
    \foreach \x in {\aa,\ab,\ac,\ad,\ae,\af,\ag} { --  (\x:\RII) node at (\x:\RII) (R2-\x) {}};
   \draw (2.36,0.63)  node {$\Diamond \bH$};
   \draw (1.16,-0.43) node {$\nabla \bH$};
   \draw (-0.22,-0.96) node {$\lnot \Square \bH$};
   \draw (-2.25,0.05) node {$\lnot \Diamond \bH$};
   \draw (-1.25,2.05) node {$\Delta \bH$};
   \draw (0.53,2.53) node {$\Square \bH$};
  \end{scope}

  \foreach \x in {\aa,\ad,\ae,\af} { \draw (R1-\x.center)--(R2-\x.center); };
  \foreach \x in {\ab,\ac} { \draw[dashed] (R1-\x.center)--(R2-\x.center); };

  \draw (\aa:\RIIFora) node {$\lnot \Square H$};
  \draw (\af:\RIIFora) node {$\nabla H$};
  \draw (\ae:\RIIFora) node {$\Diamond H$};
  \draw (\ad:\RIIFora) node {$\Square H$};
  \draw (-1.05,2.09) node {$\Delta H$};
  \draw (0.99,2.49) node {$\lnot \Diamond H$};

  \draw [color=qqqqff] [line width=0.8mm] [<->]  (R1-\aa.center)--(R2-\aa.center);
  \draw [color=qqqqff] [line width=0.8mm] [<->]  (R1-\ab.center)--(R2-\ab.center);
  \draw [color=qqqqff] [line width=0.8mm] [<->]  (R1-\ac.center)--(R2-\ac.center);
  \draw [color=qqqqff] [line width=0.8mm] [<->]  (R1-\ad.center)--(R2-\ad.center);
  \draw [color=qqqqff] [line width=0.8mm] [<->]  (R1-\ae.center)--(R2-\ae.center);
  \draw [color=qqqqff] [line width=0.8mm] [<->]  (R1-\af.center)--(R2-\af.center);
 \end{tikzpicture}
}

\newcommand{\monotonicity}{
 \tdplotsetmaincoords{70}{0}
 \begin{tikzpicture}[tdplot_main_coords,scale=1.4, every node/.style={scale=1.0}]
 \def\RI{2}
 \def\RIFora{2.7}
 \def\RII{1.25}
 \def\RIIFora{1.95}
 \def\RIII{0.01}

 \draw[thick] (\RI,0)
  \foreach \x in {0,300,240,180} { --  (\x:\RI) node at (\x:\RI) (R1-\x) {} };
 \draw[dashed,thick] (R1-0.center)
  \foreach \x in {60,120,180} { --  (\x:\RI) node at (\x:\RI) (R1-\x) {} };
 \path[fill=gray!30] (\RI,0)
  \foreach \x in {0,60,120,180,240,300} { --  (\x:\RI)};

 \begin{scope}[yshift=2cm]
 \draw[thick,fill=gray!30,opacity=0.2] (\RII,0)
  \foreach \x in {0,60,120,180,240,300,360} { --  (\x:\RII) node at (\x:\RII) (R2-\x) {}};
 \draw (0:\RIIFora) node {$\lnot \Square H$};

 \draw (180:\RIIFora) node {$\Square H$};
 \draw (120:\RIIFora) node {$\Delta H$};
 \draw (60:\RIIFora) node {$\lnot \Diamond H$};
 \end{scope}

 \foreach \x in {0,180,240,300} { \draw (R1-\x.center)--(R2-\x.center); };
 \foreach \x in {60,120} { \draw[dashed] (R1-\x.center)--(R2-\x.center); };

 \draw (0:\RIFora) node {$\lnot \Square H'$};
 \draw (300:\RIFora) node {$\nabla H'$};
 \draw (240:\RIFora) node {$\Diamond H'$};
 \draw (180:\RIFora) node {$\Square H'$};

 \draw [color=qqqqff] [line width=0.8mm] [->]  (R1-0.center)--(R2-0.center);
 \draw [color=qqqqff] [line width=0.8mm] [->]  (R1-60.center)--(R2-60.center);
 \draw [color=qqqqff] [line width=0.8mm] [<-]  (R1-240.center)--(R2-240.center);
 \draw [color=qqqqff] [line width=0.8mm] [<-]  (R1-180.center)--(R2-180.center);
\end{tikzpicture}
}

\newcommand{\ripley}{
 \tdplotsetmaincoords{70}{0}
 \begin{tikzpicture}[tdplot_main_coords,scale=0.9, every node/.style={scale=1.0}]
  \def\RI{2}
  \def\RIFora{2.7}
  \def\RII{2}
  \def\RIIFora{2.7}
  \def\RIII{2}
  \def\RIIII{3}
  \def\aa{20}
  \def\ab{80}
  \def\ac{140}
  \def\ad{200}
  \def\ae{260}
  \def\af{320}
  \def\ag{380}

  \draw[thick] (\aa:\RI)
   \foreach \x in {\aa,\af,\ae,\ad} { --  (\x:\RI) node at (\x:\RI) (R1-\x) {} };
  \draw[dashed,thick] (R1-\aa.center)
   \foreach \x in {\ab,\ac,\ad} { --  (\x:\RI) node at (\x:\RI) (R1-\x) {} };
  \path[fill=gray!30] (\RI,0)
   \foreach \x in {\aa,\ab,\ac,\ad,\ae,\af,\aa} { --  (\x:\RI)};

  \begin{scope}[yshift=2.4cm]
   \draw[thick,fill=gray!30,opacity=0.2] (\aa:\RII)
    \foreach \x in {\aa,\ab,\ac,\ad,\ae,\af,\ag} { --  (\x:\RII) node at (\x:\RII) (R2-\x) {}};
  \end{scope}

  \begin{scope}[yshift=3.8cm]
   \draw[thick,fill=gray!30,opacity=0.2] (\aa:\RII)
    \foreach \x in {\aa,\ab,\ac,\ad,\ae,\af,\aa} { --  (\x:\RII) node at (\x:\RII) (R3-\x) {}};
  \end{scope}

  \begin{scope}[yshift=6.0cm]
   \draw[thick,fill=gray!30,opacity=0.2] (\aa:\RIII)
    \foreach \x in {\aa,\ab,\ac,\ad,\ae,\af,\aa} { --  (\x:\RIII) node at (\x:\RIII) (R4-\x) {}};
    node at (\aa:\RIII) (a-top) {};
  \end{scope}

  \foreach \x in {\aa,\ad,\ae,\af} { \draw (R1-\x.center)--(R4-\x.center); };
  \foreach \x in {\ab,\ac} { \draw[dashed] (R1-\x.center)--(R4-\x.center); };

  \draw (\aa:\RIIFora) node {$\lnot \sqb H$};
  \draw (\af:\RIIFora) node {$\nbb H$};
  \draw (\ae:\RIIFora) node {$\dib H$};
  \draw (\ad:\RIIFora) node {$\sqb H$};
  \draw (-1.05,1.97) node {$\dlb H$};
  \draw (0.78,2.27) node {$\lnot \dib H$};
 
  \draw  [line width=0.8mm]   (-3.8,-0.7)--(-3.8,17);
  \draw (-3.8, -1.7) node {$p(H|x)$};
  \draw (-3.8, 6.3) node {$-$};
  \draw (-3.8, 10) node {$-$};
  \draw (-4.2, -0.3) node {$1$};
  \draw (-4.2, 6.3) node {$c_1$};
  \draw (-4.2, 10) node {$c_2$};
  \draw (-4.2, 16.8) node {$0$};
  
  \draw [color=green] [line width=0.8mm]  (R1-\ae.center)--(R2-\ae.center);
  \draw [color=green] [line width=0.8mm]  (R1-\ad.center)--(R2-\ad.center);
  \draw [color=green] [line width=0.8mm]  (R1-\ac.center)--(R2-\ac.center);
  \draw [color=red] [line width=0.8mm]   (R3-\aa.center)--(R4-\aa.center);
  \draw [color=red] [line width=0.8mm]   (R3-\ac.center)--(R4-\ac.center);
  \draw [color=red] [line width=0.8mm]   (R3-\ab.center)--(R4-\ab.center);
  
  \draw [color=qqqqff] [line width=0.8mm]   (R2-\aa.center)--(R3-\aa.center);
  \draw [color=qqqqff] [line width=0.8mm]   (R2-\ae.center)--(R3-\ae.center);
  \draw [color=qqqqff] [line width=0.8mm]   (R2-\af.center)--(R3-\af.center);
  \draw [color=qqqqff] [line width=0.8mm]   (R2-\aa.center)--(R3-\aa.center);
  \draw [color=qqqqff] [line width=0.8mm]   (R2-\ae.center)--(R3-\ae.center);
  \draw [color=qqqqff] [line width=0.8mm]   (R2-\af.center)--(R3-\af.center);
 \end{tikzpicture}
}

\newcommand{\gfbst}{
 \tdplotsetmaincoords{70}{0}
 \begin{tikzpicture}[tdplot_main_coords,scale=0.88, every node/.style={scale=0.98}]
  \def\RI{2}
  \def\RIFora{2.7}
  \def\RII{2}
  \def\RIIFora{2.7}
  \def\RIII{2}
  \def\RIIII{3}
  \def\aa{20}
  \def\ab{80}
  \def\ac{140}
  \def\ad{200}
  \def\ae{260}
  \def\af{320}
  \def\ag{380}

  \draw[thick] (\aa:\RI)
   \foreach \x in {\aa,\af,\ae,\ad} { --  (\x:\RI) node at (\x:\RI) (R1-\x) {} };
  \draw[dashed,thick] (R1-\aa.center)
   \foreach \x in {\ab,\ac,\ad} { --  (\x:\RI) node at (\x:\RI) (R1-\x) {} };
  \path[fill=gray!30] (\RI,0)
   \foreach \x in {\aa,\ab,\ac,\ad,\ae,\af,\aa} { --  (\x:\RI)};

  \begin{scope}[yshift=2.4cm]
   \draw[thick,fill=gray!30,opacity=0.2] (\aa:\RII)
    \foreach \x in {\aa,\ab,\ac,\ad,\ae,\af,\ag} { --  (\x:\RII) node at (\x:\RII) (R2-\x) {}};
  \end{scope}

  \begin{scope}[yshift=3.8cm]
   \draw[thick,fill=gray!30,opacity=0.0] (\aa:\RII)
    \foreach \x in {\aa,\ab,\ac,\ad,\ae,\af,\aa} { --  (\x:\RII) node at (\x:\RII) (R4-\x) {}};
  \end{scope}

  \begin{scope}[yshift=6.0cm]
   \draw[thick,fill=gray!30,opacity=0.2] (\aa:\RIII)
    \foreach \x in {\aa,\ab,\ac,\ad,\ae,\af,\aa} { --  (\x:\RIII) node at (\x:\RIII) (R3-\x) {}};
    node at (\aa:\RIII) (a-top) {};
  \end{scope}

  \foreach \x in {\aa,\ad,\ae,\af} { \draw (R1-\x.center)--(R2-\x.center); };
  \foreach \x in {\ab,\ac} { \draw[dashed] (R1-\x.center)--(R2-\x.center); };
  \foreach \x in {\ab,\ac} { \draw[dashed] (R2-\x.center)--(R3-\x.center); };
  \foreach \x in {\aa,\af,\ae,\ad} { \draw (R2-\x.center)--(R3-\x.center); };

  \draw (\aa:\RIIFora) node {$\lnot \Square H$};
  \draw (\af:\RIIFora) node {$\nabla H$};
  \draw (\ae:\RIIFora) node {$\Diamond H$};
  \draw (\ad:\RIIFora) node {$\Square H$};
  \draw (-1.15,1.95) node {$\Delta H$};
  \draw (0.83,2.23) node {$\lnot \Diamond H$};

  \draw (-3.85,-0.3) node {$c=1$};
  \draw (-3.85,16.8) node {$c=0$};
\draw (-4.3,6.5) node {$\ev(\bH|x)$};
\draw (-3.2, 6.5) node {$\boldmath{-}$};

  \draw [color=green] [line width=0.8mm]    (R2-\aa.center)--(R3-\aa.center);
  \draw [color=qqqqff] [line width=0.8mm]   (R1-\ae.center)--(R2-\ae.center);
  \draw [color=green] [line width=0.8mm]   (R2-\ae.center)--(R3-\ae.center);
  \draw [color=qqqqff] [line width=0.8mm]   (R1-\ad.center)--(R2-\ad.center);
  \draw [color=qqqqff] [line width=0.8mm]   (R1-\ac.center)--(R2-\ac.center);
  \draw [color=green] [line width=0.8mm]   (R2-\af.center)--(R3-\af.center);
  \draw  [line width=0.8mm]   (-3.2,-0.7)--(-3.2,17);
 \end{tikzpicture}
 \hspace{-3mm}
 \begin{tikzpicture}[tdplot_main_coords,scale=0.88, every node/.style={scale=0.98}]
  \def\RI{2}
  \def\RIFora{2.7}
  \def\RII{2}
  \def\RIIFora{2.7}
  \def\RIII{2}
  \def\RIIII{3}
  \def\aa{20}
  \def\ab{80}
  \def\ac{140}
  \def\ad{200}
  \def\ae{260}
  \def\af{320}
  \def\ag{380}

  \draw[thick] (\aa:\RI)
   \foreach \x in {\aa,\af,\ae,\ad} { --  (\x:\RI) node at (\x:\RI) (R1-\x) {} };
  \draw[dashed,thick] (R1-\aa.center)
   \foreach \x in {\ab,\ac,\ad} { --  (\x:\RI) node at (\x:\RI) (R1-\x) {} };
  \path[fill=gray!30] (\RI,0)
   \foreach \x in {\aa,\ab,\ac,\ad,\ae,\af,\aa} { --  (\x:\RI)};

  \begin{scope}[yshift=2.4cm]
   \draw[thick,fill=gray!30,opacity=0.2] (\aa:\RII)
    \foreach \x in {\aa,\ab,\ac,\ad,\ae,\af,\ag} { --  (\x:\RII) node at (\x:\RII) (R2-\x) {}};
  \end{scope}

  \begin{scope}[yshift=3.3cm]
   \draw[thick,fill=gray!30,opacity=0.0] (\aa:\RII)
    \foreach \x in {\aa,\ab,\ac,\ad,\ae,\af,\aa} { --  (\x:\RII) node at (\x:\RII) (R4-\x) {}};
  \end{scope}

  \begin{scope}[yshift=6.0cm]
   \draw[thick,fill=gray!30,opacity=0.2] (\aa:\RIII)
    \foreach \x in {\aa,\ab,\ac,\ad,\ae,\af,\aa} { --  (\x:\RIII) node at (\x:\RIII) (R3-\x) {}};
    node at (\aa:\RIII) (a-top) {};
  \end{scope}

  \foreach \x in {\aa,\ad,\ae,\af} { \draw (R1-\x.center)--(R2-\x.center); };
  \foreach \x in {\ab,\ac} { \draw[dashed] (R1-\x.center)--(R2-\x.center); };
  \foreach \x in {\ab,\ac} { \draw[dashed] (R2-\x.center)--(R3-\x.center); };
  \foreach \x in {\aa,\af,\ae,\ad} { \draw (R2-\x.center)--(R3-\x.center); };

  \draw (\aa:\RIIFora) node {$\lnot \Square \bH$};
  \draw (\af:\RIIFora) node {$\nabla \bH$};
  \draw (\ae:\RIIFora) node {$\Diamond \bH$};
  \draw (\ad:\RIIFora) node {$\Square \bH$};
  \draw (-1.15,1.95) node {$\Delta \bH$};
  \draw (0.83,2.23) node {$\lnot \Diamond \bH$};

  \draw [color=green] [line width=0.8mm]    (R2-\aa.center)--(R3-\aa.center);
  \draw [color=green] [line width=0.8mm]   (R2-\ae.center)--(R3-\ae.center);
  \draw [color=red] [line width=0.8mm]   (R1-\aa.center)--(R2-\aa.center);
  \draw [color=red] [line width=0.8mm]   (R1-\ac.center)--(R2-\ac.center);
  \draw [color=red] [line width=0.8mm]   (R1-\ab.center)--(R2-\ab.center);
  \draw [color=green] [line width=0.8mm]   (R2-\af.center)--(R3-\af.center);

 \end{tikzpicture}
}

\newcommand{\gfbstfig}{
 \begin{tikzpicture}[thick,scale=0.75, every node/.style={scale=0.9}]
  \draw (-6,2) circle (1.5);
  \draw [fill=lightgray] (-6,2) circle (0.8);
  \draw (-8,0) -- (-4,0) -- (-4,4) -- (-8,4) -- (-8,0);;
  \node at (-6,2) {{\large$S$}}; 
  \node at (-6,0.8) {{\large$H$}};
  \node at (-7.5,.50) {{\large$\bH$}};
  \node at (-6,4.25) {{\large 	$\mbox{A}: \square H$}};
			
  \draw (-1.4,2.4) circle (1.1);
  \draw [fill=lightgray] (0.2,0.9) circle (0.7);
  \draw (-3,0) -- (1,0) -- (1,4) -- (-3,4) -- (-3,0);;
  \node at (0.2,0.9) {{\large$S$}}; 
  \node at (-1.4,2.4) {{\large$H$}};
  \node at (-2.5,.50) {{\large$\bH$}}; 
  \node at (-1,4.25) {{\large $\mbox{E}: \lnot \Diamond  H$}};	
			
  \draw [fill=lightgray] (5.2,1.9) circle (0.7); 
  \draw (4.4,2.4) circle (1.1);
  \draw (2,0) -- (6,0) -- (6,4) -- (2,4) -- (2,0);;
  \node at (5.2,1.9) {{\large$S$}}; 
  \node at (4.2,2.5) {{\large$H$}};
  \node at (2.5,.50) {{\large$\bH$}};
  \node at (4,4.25) {{\large $	\mbox{Y}: \nabla H$}};	
 \end{tikzpicture} \\ 
}

\DeclareRobustCommand{\rchi}{{\mathpalette\irchi\relax}}
\newcommand{\irchi}[2]{\raisebox{\depth}{$#1\chi$}} 

\begin{document}   

\maketitle

\begin{abstract}
Although logical consistency is desirable in scientific research,  standard statistical hypothesis tests are typically logically inconsistent. 
In order to address this issue, 
previous work introduced agnostic hypothesis tests and proved that they 
can be logically consistent while retaining statistical optimality properties. This paper
characterizes the credal modalities in agnostic hypothesis tests and uses
the hexagon of oppositions to explain 
the logical relations between these modalities. Geometric solids that are composed of hexagons of oppositions
illustrate the conditions for 
these modalities to be 
logically consistent.
Prisms composed of hexagons of oppositions show how the credal modalities obtained from two agnostic tests vary according to their threshold values. Nested hexagons of oppositions summarize logical relations between the credal modalities in these tests and
prove new relations.
\end{abstract}



\section{Introduction}

Logical consistency is desirable in scientific research. 
It can be hard to interpret inconsistent conclusions and
to develop theories that explain them.
As a result, they are detrimental to 
scientific communication and decision-making.

Despite the desirability of logical consistency,
it is not satisfied by many methods of statistical hypothesis testing.
For instance, consider that a researcher simultaneously tests two hypotheses, $A$ and $B$, 
such that ``not $B$'' implies ``not $A$''.
In such a situation, one expects that the rejection of $B$ would imply the rejection of $A$.
However, this property does not apply to standard tests,
such as the ones based on
the p-value or the Bayes factor \citep{Schervish1996, Schervish1999, Fossaluza2016}.

Besides this flaw, 
one can find other 
logical inconsistencies in 
standard hypothesis testing procedures.
\citet{Izbicki2015,Silva2015} classify
the types of logical inconsistencies in
these procedures into failures to
satisfy at least one of four of
conditions of classical logics.
If all conditions are satisfied by
a test, then it is defined as 
logically consistent.
This use of the term contrasts with
the one in the literature of 
non-classical logics, in which
(in)consistency often refers to 
forms of modifying classical notions of contradiction, negation, conflation, etc. so as to allow logical systems with
greater flexibility than classical logic \citep{Carnielli2007,Carnielli2016,Stern2004}.
\citet{Izbicki2015} shows that 
a standard test is
logically consistent if and only if 
it is based on point estimation,
which generally does not satisfy statistical optimality. 

This grim result motivated \citet{Esteves2016}
to investigate \textit{agnostic hypothesis tests}. Such a test is a
function, $\mathcal{L}$, that assigns to
each hypothesis, $H$, a value in 
$\{0,0.5,1\}$. $H$ is accepted when
$\mathcal{L}(H)=0$, $H$ is rejected when
$\mathcal{L}(H)=1$ and the test 
noncomitally neither accepts nor 
rejects $H$ when $\mathcal{L}(H)=0.5$.
In this paper, 
we use agnostic tests to 
define \textit{credal modalities},
as described in
\Cref{table:modalities}.

\begin{table}
 \centering
 \resizebox{\textwidth}{!}{
 \begin{tabular}{|c|l|l|c|l|}
  \hline
  Modality 
  & Definition 
  & Name 
  & Equivalence 
  & Interpretation \\
  \hline
  $\square H$ 
  & $\mathcal{L}(H) = 0$ 
  & Necessity (A) 
  & $\Delta H \wedge \Diamond H$ 
  & $H$ is accepted \\
  $\lnot \Diamond H$ 
  & $\mathcal{L}(H) = 1$ 
  & Impossibility (E) 
  & $\Delta H \wedge \lnot \square H$ 
  & $H$ is rejected \\
  $\nabla H$ 
  & $\mathcal{L}(H) = 0.5$
  & Contingency (Y)
  & $\Diamond H \wedge \lnot \square H$ 
  & $H$ not decided \\
  $\Diamond H$ 
  & $\mathcal{L}(H) < 1$
  & Possibility (I) 
  & $\square H \vee \nabla H$ 
  & $H$ not rejected \\
  $\lnot \square H$ 
  & $\mathcal{L}(H) > 0$ 
  & Non-necessity (O) 
  & $\lnot \Diamond H \vee \nabla H$ 
  & $H$ not accepted \\
  $\Delta H$ 
  & $\mathcal{L}(H) \neq 0.5$ 
  & Non-contingency (U) 
  & $\square H \vee \lnot \Diamond H$ 
  & $H$ is decided \\
  \hline
 \end{tabular} 
 }
 \caption{Modalities of 
 agnostic hypothesis tests.}
 \label{table:modalities}
\end{table}

Using this framework, 
\citet{Esteves2016} shows that the agnostic generalizations
of some standard hypothesis tests can achieve logical consistency.
Indeed, a hypothesis test is logically consistent if and only if 
it is based on a region estimator.
As a result, there exist agnostic hypothesis tests that
are both logically consistent and statistically optimal.

\begin{figure}
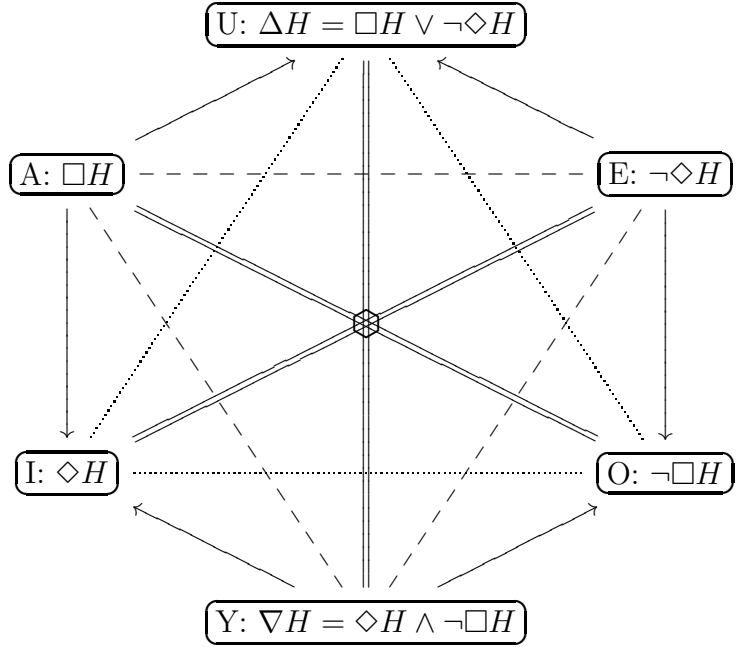

 \centering
 \hexagono
 \caption{The hexagon of oppositions for an agnostic hypothesis test.}
 \label{figure:hexa-agnostic}
\end{figure}

This positive result motivates 
the development of schematic diagrams for
explaining agnostic tests and
their logical properties.
The relations between the 
credal modalities in 
\Cref{table:modalities} can 
be represented by 
the hexagon of oppositions
\citep{Blanche1953, Blanche1981}.
It follows from the definition 
in \Cref{table:modalities} that
these credal modalities are defined by 
opposition relations in 
well ordered structures, 
such as the ones in 
\citet{Beziau2012, Beziau2015}. 
Although this paper considers solely total orders, extensions of 
the hexagon of oppositions
to partial orders can be found in
\citet{Demey2016}. 
    
 \Cref{figure:hexa-agnostic} represents
the standard \emph{hexagon of oppositions}.
Each credal modality is associated with 
one vertex in the logical hexagon, namely:
A (necessity), E (impossibility),
I (possibility), O (non-necessity),
U (non-contingency) and
Y (contingency).\footnote{ 
 The four vowels used to label the  vertices of the traditional square of opposition are chosen from the Latin word \textit{AfIrmo}, for affirming universal (necessary) and particular (possible) modal propositions,  and \textit{nEgO} for negating them. 
\cite{Blanche1953} extended the square into the hexagon of opposition, using the two remaining French vowels to label the top and bottom of its six vertices.  
 Some confusion comes from the historical use of four vowels (A, E, I and U) to describe which part(s) of a modal  proposition (modus and dictum) are  negated. 
 This description follows the mnemonic rule: \textit{nihil A, E dictum negat, Ique modum,  U totum}, meaning: 
  A- neither modus nor dictum is negated, 
  E- dictum is negated
  I- modus is negated, 
  U- modus and dictum are negated.    
 Finally, each corner of the square of opposition is traditionally labeled by an anagram describing the (four) equivalent modal propositions obtained by using, respectively, the modality of --  necessary, impossible, contingent and possible.   
  Mnemonic anagrams are provided by the Latin words:  
 (A) Purpurea - purple,      
 (E) Iliace - colicky,   
 (I) Amabimus - lovely, and     
 (O) Edentuli - no-tooth; see 
 \cite[v.II, p.109-111]{Dumitriu1977}.  
} 
In \Cref{figure:hexa-agnostic}, logical relations between modalities 
are represented as edges, which are classified according to 
the following types:

\begin{itemize}
 \item arrows: indicate \textit{logical implications}, that is,
 if modality $J$ points to $K$ and $J$ is obtained,
 then $K$ is also obtained.

 \item dashed-lines: indicate \textit{contrariety}, that is,
 modalities connected by these lines cannot both be true.

 \item dotted-lines: indicate \textit{sub-contrariety}, that is,
 modalities connected by these lines cannot both be false. 
  
 \item double-lines: indicate \textit{contradictions}, that is,
 modalities connected by these lines cannot either 
 be both true or be both false.
\end{itemize}

This paper characterizes
the credal modalities from the
statistical theory developed in
\citet{Esteves2016} and
uses the hexagon of oppositions
to explain the logical properties that
arise in these modalities.
\Cref{section:properties} uses
polyhedra based on
the hexagon of oppositions
to present the four conditions for 
a test to be logically consistent.
Next, hexagonal prisms illustrate the
construction of the modalities
that arise from standard hypothesis tests (\Cref{section:standard-tests}) and that arise from a logically
consistent agnostic test (\Cref{section:region-tests}).
\Cref{section:hybrid} uses a nested hexagon of oppositions to discuss the logical relations between the modalities obtained in
\Cref{section:standard-tests,section:region-tests}.

Note that \Cref{section:standard-tests,section:region-tests} examine
two different types of credal modalities. 
In order to avoid ambiguity, the credal modality defined in 
each section is represented by a different symbol.
\Cref{section:standard-tests} defines
credal modalities based on threshold values on 
a \textit{probability measure} over the space of hypotheses.
Since probability is 
an additive measure, 
the corresponding 
\textit{probabilistic modalities} 
are represented by superposing the 
plus sign to the 
standard modal symbols  
($\sqb$, $\dib$, $\nbb$ and $\dlb$). 
In contrast, 
Section \ref{section:region-tests}
defines credal modalities
based on threshold values on a
\textit{possibilistic measure}
over the space of hypotheses. 
 The corresponding \textit{alethic modalities} are denoted by the standard modal symbols 
($\square$, $\Diamond$, $\nabla$ and $\Delta$).  
For a detailed explanation of probability and possibility measures see
\citet{Dubois1982,Borges2007,Stern2014}.

\section{Logical conditions on agnostic tests}
\label{section:properties}

In order to discuss the
logical conditions on statistical tests,
it is necessary to present some definitions.
A \textit{parameter} is an unknown or unobservable quantity
with possible values in an arbitrary set, $\Theta$,
named the \textit{parameter space}.
A \textit{statistical hypothesis}, $H$, is a statement about 
a parameter. It is of the form 
$H: \theta \in \Theta_H$,
where $\Theta_H$ is an element of
an arbitrary $\sigma$-field over $\Theta$.
For example, 
the parameter space could be 
$\Theta = \mathbb{R}^{d}$, $d \in \mathbb{N}^{*}$, and 
a hypothesis could be 
stated as a linear equation, $H: A\theta =b$, 
corresponding to the hypothesis set 
$\Theta_H = \{\theta \in \mathbb{R}^{d} \g A\theta =b \}$.
Whenever there is no ambiguity,
the symbol $H$ is used as 
a shortcut for $\Theta_H$,
although the former is a statement and
the latter is a set.
A \textit{statistical hypothesis test} is a
 function that attributes credal modalities
to each statistical hypothesis.

\citet{Esteves2016} introduces 
four logical consistency conditions
for statistical hypothesis tests: 
invertibility, monotonicity, union consonance and 
intersection consonance.
In the following, these properties are represented using geometric solids composed of hexagons of oppositions.

\subsection{Invertibility}

\textit{Invertibility} restrains the conclusions that can be obtained
when simultaneously testing a hypothesis, $H$, and its set complement, $\bH = \Theta - H$. 
If invertibility holds, then 
either both hypotheses are undecided or
one is accepted and the other is rejected.
This restriction is represented symbolically in \Cref{defn:invertibility}:

\begin{definition}[Invertibility]
 \label{defn:invertibility} 
 $$(\square H  \iff  \lnot \Diamond \bH) 
   \ \ \wedge \ \  (\nabla H  \iff   \nabla  \bH)$$
\end{definition} 
 
\Cref{fig::invertibility} uses two hexagons of opposition
to represent \Cref{defn:invertibility}.
The hexagons for $H$ (below) and $\bH$ (above) are 
joined by double-arrowed edges to form a hexagonal prism.
These edges illustrate that, if invertibility holds,
then the connected modalities necessarily occur simultaneously.
The modalities in the $\bH$ hexagon are inverted
in relation to the modalities in the $H$ hexagon.
As a result, if invertibility holds, then
one obtains the following logical equivalences:
$A \leftrightarrow \mbox{\bE}$, 
$E \leftrightarrow \mbox{\bA}$, 
$I \leftrightarrow \mbox{\bO}$, 
$O \leftrightarrow \mbox{\bI}$, 
$U \leftrightarrow \mbox{\bU}$,   
$Y  \leftrightarrow \mbox{\bY}$. 

\begin{figure}
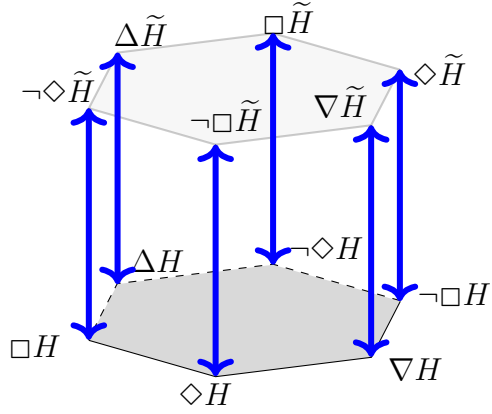

\centering
\invertibility
\caption{A representation of invertibility using two hexagons of oppositions.
         $\bH$ denotes the negation of a hypothesis $H$.
		 Two modalities are connected by a bidirectional arrow if 
		 each one implies and is implied by the other.}
\label{fig::invertibility}
\end{figure}

\subsection{Monotonicity}

\textit{Monotonicity} restrains the conclusions that can be obtained when
simultaneously testing nested hypotheses.
Two hypotheses, $H$ and $H'$, are nested if
$H$ implies $H'$ ($H \subseteq H'$).
If monotonicity holds, then the conclusion obtained for $H'$ is
always at least as favorable as the conclusion obtained for $H$.
That is, if $H$ is accepted, then $H'$ is accepted and, 
if $H$ is possible, then $H'$ is possible.
This restriction is represented symbolically in \Cref{defn:monotonicity}.

\begin{definition}[Monotonicity]
 \label{defn:monotonicity}  
 \begin{align*} 
  H \subseteq H' \Rightarrow 
  \begin{cases}
   \square H  \Rightarrow  \square H'  \\
   \Diamond H  \Rightarrow  \Diamond H'
  \end{cases}
 \end{align*}
\end{definition}	

\Cref{fig::monot} represents monotonicity through
the hexagon of oppositions.
The hexagon for $H'$ is larger than the one for $H$
as a means to illustrate that $H \subseteq H'$.
These hexagons are connected by additional edges to 
form a pyramidal frustum.
Each horizontal cut of this solid can be thought to represent a new hypothesis.
These hypotheses imply the ones below them and are implied by the ones above them.
The arrows that go from the top of the frustum to its bottom represent that,
if a hypothesis is below another, 
then the conclusion obtained for the lower one
should be at least as favorable as the one
obtained for the upper one.
The arrows that go in the opposite direction are 
obtained by taking the negation of 
the previous implications.
As a result of monotonicity, one obtains the following
logical implications:    
$\mbox{A} \rightarrow \mbox{A'}$, 
$\mbox{I} \rightarrow \mbox{I'}$, 
$\mbox{O'}\rightarrow \mbox{O}$, 
$\mbox{E'}\rightarrow \mbox{E}$. 

\begin{figure}
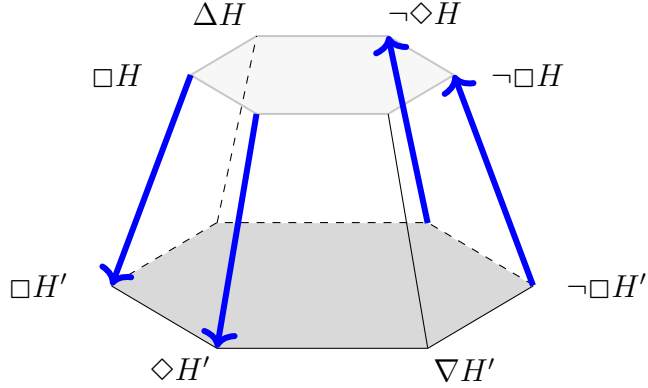

 \centering
 \monotonicity
 \caption{Logical implications induced by monotonicity (blue arrows). The figure displays
stacked hexagons of statistical modalities for a sequence of nested hypotheses. Each horizontal cut represents a different hypothesis; in the figure we display two of them: $H \subseteq H'$.  }
 \label{fig::monot}
\end{figure}

\subsection{Consonance}

\textit{Consonance} restricts the conclusions that one can obtain
when simultaneously testing a set of hypotheses and 
their unions or intersections.
Let $\{ H_{i} \}_{i \in I}$ be an arbitrary collection of hypotheses.
Under strong union consonance, 
if one concludes that $\cup_{i \in I}{H_{i}}$ is possible,
then at least one $H_{i}$ is possible.
Similarly, under strong intersection consonance,
if one concludes that $\cap_{i \in I}{H_{i}}$ is not necessary,
then at least one $H_{i}$ is not necessary.
These restrictions are represented symbolically in\Cref{defn:union_consonance,defn:inter_consonance}.

\begin{definition}[Strong union consonance]
 \label{defn:union_consonance}
 For every $I$,
 \begin{align*}
  \Diamond  ( \cup_{i \in I} H_i)  \rightarrow
  \exists i \in I \mbox{ s.t. } \Diamond  H_i.
 \end{align*}
\end{definition}
Strong union consonance is equivalent to
\begin{align*}
 \forall i \in I, \lnot \Diamond H_i \Rightarrow
 \lnot \Diamond (\cup_{i \in I}H_i)
\end{align*}
that is, the $\lnot \Diamond$-modality is
closed under the union.

\begin{definition}[Strong intersection consonance]
 \label{defn:inter_consonance}
 For every $I$,
 \begin{align*}
  \lnot \square  (\cap_{i \in I} H_i) \Rightarrow
  \exists i \in I \mbox{ s.t. } \lnot \square  H_i.
 \end{align*}
\end{definition}
Strong intersection consonance is
equivalent to
\begin{align*}
 \forall i \in I, \square H_i \Rightarrow 
 \square ( \cap_{i \in I} H_i),
\end{align*}
that is, the $\square$-modality is 
closed under the intersection.
 
\Cref{fig:consonance} illustrates\Cref{defn:union_consonance,defn:inter_consonance}.
For union consonance, consider a triangle in which each vertex represents the possibility of a single hypothesis.
If one obtains the center of the triangle and the union of all hypotheses is possible,
then one also obtains at least one of the midpoints of the edge, representing the possibilities of pairwise unions, and
at least one of the vertices connected to this edge.
\Cref{fig:consonance} also illustrates the symmetry between
strong union consonance and strong intersection consonance.
Indeed, if invertibility holds,
then these types of consonance are equivalent.

\begin{figure}
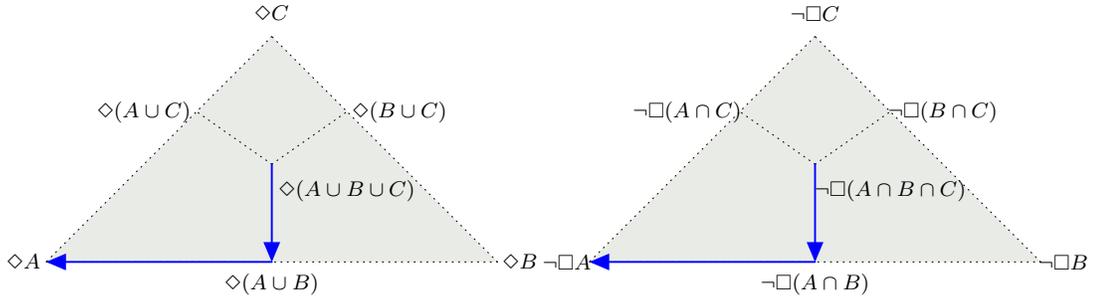

 \centering
 \begin{subfigure}{.52\textwidth}
  \unionconsonance
 \end{subfigure}%
 \begin{subfigure}{.52\textwidth}
  \interconsonance
 \end{subfigure}
 \caption{Logical implications induced by strong union consonance (left) and strong intersection consonance (right). If one obtains the conclusion in the center of the triangle, then one obtains at least one of the conclusions in the middle of the edges of the triangle and another in the vertices connected to this edge. That is, there exists at least one path such as the one illustrated by the blue arrows.}
 \label{fig:consonance}
\end{figure}

\subsection{The logic of credal modalities in logically consistent hypothesis tests}

A hypothesis test is \textit{logically consistent} if it satisfies
invertibility, monotonicity, union and intersection consonance, and
$\square \Theta$ is obtained.
Logically consistent tests have properties that improve
their computation and interpretation.
For example, it follows from 
invertibility and \Cref{table:modalities} that,
for every hypothesis $H$,
$\square H$ and $\square \bH $ cannot be
simultaneously obtained.
That is, a logically consistent hypothesis test
admits no contradictions.

Logically consistent tests also allow
the use of classical logic to
obtain deductions over decided hypotheses.
In order to show this property, 
we use the functionally complete ``nand'' 
operator ($\uparrow$), either as defined over  truth values or as defined over hypotheses (sets). 
Over hypotheses, 
$H_1 \uparrow H_2 := \Theta - (H_1 \cap H_2)$.
It follows that, for example, 
$H_1 \wedge H_2 = H_1 \cap H_2$,
$H_1 \vee H_2 = H_1 \cup H_2$ and 
$\neg H_1 = \Theta - H_1$.
Over truth values, 
$\square H_1 \uparrow \square H_2$ is obtained if
and only if either $\square H_1$ is not obtained or
$\square H_2$ is not obtained.
\Cref{lemma:logical} in the appendix shows that,
if a test is logically consistent and
$H_1$ and $H_2$ are decided 
($\Delta H_1$ and $\Delta H_2$ are obtained), then
$\square(H_1 \uparrow H_2)$ is obtained if and only if
$\square H_1 \uparrow \square H_2$ is obtained.
In other words, if a test is logically consistent and
$H_1,\ldots,H_n$ are decided, then, 
for every logical proposition, $P$,
$\square P(H_1,\ldots,H_n)$ is obtained if and only if
$P(\square H_1,\ldots, \square H_n)$ is obtained.
Once one has decided the credal modality of
$H_1,\ldots, H_n$, 
computing $P(\square H_1,\ldots, \square H_n)$
is usually cheaper than computing $\square P(H_1,\ldots,H_n)$
by directly applying the statistical test.
Also, this ability to treat decided hypothesis
as if they were true or false makes the outcomes of
the test easier to interpret.

Besides the application of classical logic to
decided hypotheses, one can also use it between
different credal modalities.
For example, the implications in
\Cref{figure:hexa-agnostic,fig::invertibility,fig::monot,fig:consonance} are transitive.
This can be used to obtain new relations by combining
the implications in each of the figures.
For example, it follows from \Cref{fig::invertibility} that
$(\neg \Diamond H) \rightarrow (\square \bH )$.
Also, it follows from \Cref{figure:hexa-agnostic} that
$(\square \bH ) \rightarrow (\Diamond \bH )$. Therefore, 
one can conclude that $(\neg \Diamond H) \rightarrow (\square \bH )$.
Similarly, it follows from \Cref{fig::invertibility} that 
$(\square H) \rightarrow \neg (\Diamond \bH )$.
Also, it follows from the contrariety relation in \Cref{figure:hexa-agnostic}
that is obtained from \Cref{table:modalities} that
$(\Diamond \bH ) \rightarrow (\neg \nabla \bH )$.
By combining these results one can conclude that,
if $\square H$ is obtained, then 
$\nabla \bH $ is not obtained. 

Despite the advantages of working with
logically consistent hypothesis tests, 
several standard tests available in 
the statistical literature are 
logically inconsistent. 
Indeed, standard hypothesis tests can 
fail basic logical properties.
The following section investigates 
examples of logical inconsistency in 
frequentist and Bayesian statistics.

\section{Logical inconsistency of standard tests}
\label{section:standard-tests}

Hypothesis tests are often used to 
determine beliefs. Using \Cref{table:modalities}, 
one can translate each possible result of an
hypothesis test to a credal modality.
In case a hypothesis test does not satisfy all 	
the properties in \Cref{section:properties}, then 
we say the resulting logical system is incoherent.
Even though it is challenging to communicate and interpret these systems,
they are the ones that are generally obtained from standard hypothesis tests.
The next subsection revisits some known examples of 
logical inconsistencies in classical tests.

\subsection{Classical tests}

\begin{itemize}
 \item \textbf{Failure of invertibility}: Tests based on p-values treat
 the null and the alternative hypotheses differently.
 Indeed, the p-value is calculated using solely probabilities
 obtained under the null hypothesis;
 probabilities obtained under alternative hypotheses
 are left out.
 As a result, a test based on a p-value generally will not be invertible.
 
 \item \textbf{Failure of monotonicity}: Under regularity conditions, one can assume that if two individuals have the same genotype, then they will also have the same corresponding phenotypical characteristics.
However, \citet{Izbicki2012} shows that the
likelihood ratio test can reject the hypothesis that
two individuals have the same phenotype and not reject
the hypothesis that they have the same genotype.

\item \textbf{Failure of consonance}: 
Consider the standard framework of analysis of variance with three groups.
Let $\alpha_{1}$, $\alpha_{2}$ and $\alpha_{3}$
be the unknown means of each group.
\citet{Izbicki2015} shows that,
using the generalized likelihood ratio test,
one can conclude that 
``$\alpha_{1}=\alpha_{2}=\alpha_{3}=0$'' is impossible and, at the same time, 
conclude that ``$\alpha_{i}=0$'' is possible, 
for every $i \in \{1,2,3\}$.
\end{itemize} 

As an alternative to the above classical procedures,
one might consider a Bayesian hypothesis test.
One such hypothesis test
obtains the credal modality for a hypothesis
from its posterior probability.
Specifically, the credal modality is chosen according to cutoff levels
of the posterior probability.
For simplicity, these tests are called
``posterior probability cutoff tests''.
Since posterior probability tests are coherent from the perspective of Statistical Decision Theory, 
one might expect that they are also logically coherent.
This is not the case, as shown in the following section.

\subsection{Posterior probability cutoff tests}

Posterior probabilities can be used to define 
\textit{probabilistic modal operators} 
($\sqb$, $\dib$, $\nbb$ and $\dlb$), 
as explained in the following paragraphs. 
We use these special symbols to distinguish the 
probabilistic operators from the 
standard (alethic) operators that were used so far: $\square$, $\Diamond$, $\nabla$ and $\Delta$.  
    
A Bayesian statistical model considers two types of random variables:  a parameter $\theta$ in the parameter space $\Theta$,
and data, $X$, in the sample space, $\rchi$. 
While parameters correspond to unknown quantities, 
the data correspond to observable or observed quantities.  
 
The uncertainty about the model's variables
is described through probability statements.
Before observing the data, 
the joint distribution for $\theta$ and $X$ is 
denoted by $p(\theta,x)$.
By integrating $x$ out of the joint distribution, 
one obtains
\begin{align*}
 p(\theta)	&= \int_{\rchi}{p(\theta,x)dx}
\end{align*}
The distribution $p(\theta)$ is called the 
{\it prior} distribution for $\theta$,  
since it represents the uncertainty about $\theta$ 
before observing $X$.
 
After the value $x$ of $X$ is observed, 
the uncertainty about $\theta$ is updated
based on this observation.
One's uncertainty about $\theta$ changes from $p(\theta)$
to $p(\theta \g x)$, the probability of $\theta$ given $x$.
The latter term is called the 
{\it posterior} distribution for $\theta$
and can be obtained using Bayes Theorem:
\begin{align}
 \label{eq:bayes}
 p(\theta \g x)	&= \frac{p(\theta)p(x \g \theta)}{\int_{\Theta} p(\theta,x)d\theta}  
\end{align}
  
 \begin{table}
  \centering
  \begin{tabular}{|c|cc|}
   \hline
   \diagbox{Decision}{Truth} & $H$  & $\lnot H$	 \\
   \hline
   $\sqb H$ 	             & 0    & 1		     \\
   $\lnot \dib H$ 	     & a    & 0		     \\
   \hline 
  \end{tabular}
  \caption{Example of a loss function for a Bayesian hypothesis test.}
  \label{table:01a-loss}
 \end{table}

In order to attribute a credal modality to 
a hypothesis under scrutiny, 
one can use Bayesian decision theory \citep{Degroot2005}.
In this context, one chooses a credal modality
by minimizing the expected penalty that 
derives from a given \emph{loss function}.
\Cref{table:01a-loss} presents a loss function that
is typically used in Bayesian hypothesis tests.
If $H$ is true and one decides that 
$H$ is necessary, then the loss is $0$.
Similarly, if $H$ is false and one decides that
$H$ is impossible, then the loss is $0$.
However, if $H$ is true and one decides that 
$H$ is impossible, then the loss is $a$, where $a > 0$.
This type of decision is called a type-1 error.
Similarly, if $H$ is false and one decides that 
$H$ is necessary, then the loss is $1$.
This type of decision is called a type-2 error.
The constant $a$ defines how much a type-1 error is worse than
a type-2 error.
When $a=1$, both types of error are equally undesirable.
 
The optimal decision for 
the loss function in \Cref{table:01a-loss} is 
to choose $\sqb H$, when 
$p(H \g x) > (1+a)^{-1}$, and
$\lnot \dib H$, when 
$p(H \g x) < (1+a)^{-1}$.
In other words, the credal modality for 
each hypothesis is decided based on 
a cutoff of the posterior probability.
If the posterior probability of 
a hypothesis is sufficiently high,
then one concludes that it is necessary.
Otherwise, one concludes that 
the hypothesis is impossible.    
 
\begin{table}
 \centering
 \begin{tabular}{|c|cc|}
  \hline
  \diagbox{Decision}{Truth} & $H$  & $\lnot H$	 \\
  \hline
  $\sqb H$        & 0    & 1 \\
  $\nbb H$        & b    & b \\
  $\lnot \dib H$  & a    & 0 \\
  \hline
 \end{tabular}
 \caption{Example of a loss function for 
 an agnostic Bayesian hypothesis test.}
 \label{table:01ab-loss}
\end{table}
 
The loss function in \Cref{table:01a-loss} can
be extended to agnostic hypothesis tests,
as presented in \Cref{table:01ab-loss}.
In agnostic tests, one can choose to 
let the hypothesis be undecided,
that is, one can choose $\nbb H$.
Whenever this option is chosen, 
one commits a type-3 error and
incurs in a loss of $b$, where 
$0 < b < \min(a,1)$. 
This range of values for $b$ implies that,
by remaining agnostic, one loses
less than by taking a wrong decision and 
more than by taking a correct decision.
 
The optimal decision for this loss function is
similar to the one that was discussed for standard tests.
By taking $c_{1} = \max((1+a)^{-1},1-b)$ and $c_{2} = \min((1+a)^{-1},\frac{b}{a})$,
one obtains that the optimal test is:
\begin{align}
 \label{eqn:cutoff-test}
 \text{Choose credal modality}
 \begin{cases}
  \sqb H       & \text{, if } p(H \g x) > c_{1} \\
  \lnot \dib H & \text{, if } p(H \g x) < c_{2} \\
  \nbb H	   & \text{, otherwise.}
 \end{cases}
\end{align}
In words, the probabilistic modal operator is obtained by
comparing the posterior probability of the hypothesis, $H$, to 
the cutoffs, $c_{1}$ and $c_{2}$.
If $p(H \g x)$ is larger than the upper cutoff, $c_{1}$,
then the optimal decision is that $H$ is necessary.
If $p(H \g x)$ is smaller than the lower cutoff, $c_{2}$,
then the optimal decision is that $H$ is impossible.
Finally, for intermediate values of $p(H \g x)$,
$H$ remains undecided.
This test is illustrated in \Cref{fig:ripley}.

\begin{figure}
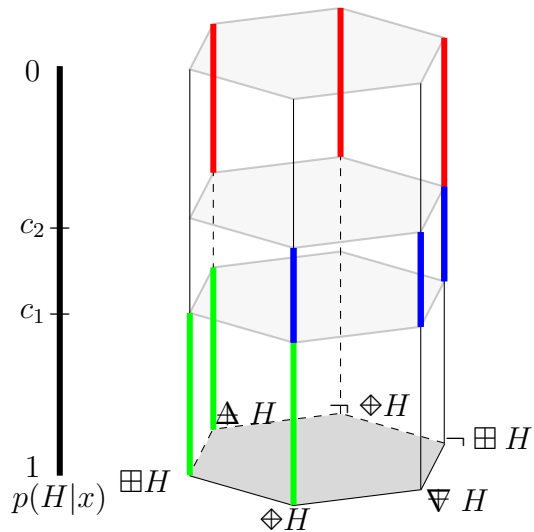

\centering
\ripley
\caption{Typical hexagon of statistical modalities for
         the posterior probability cutoff test, 
         given by \Cref{table:01ab-loss}.
         The colored edges indicate the chosen modalities
         as a function of the posterior probability of $H$,
         $p(H \g x)$.}
\label{fig:ripley}
\end{figure}

The posterior probability cutoff tests are logically incoherent.
Although the values of $a$ and $b$ can be chosen so that 
the tests are invertible and monotonic, 
these tests are not consonant.
The main argument for the lack of consonance can 
be presented in two steps.
First, the hypothesis $H_{0}: \theta \in \Theta$ 
has probability $1$.
Second, $H_{0}$ can generally be 
partitioned into a large collection of hypotheses,
$H_{1},\ldots,H_{n}$, such that 
each $P(H_{i} \g x)$ is arbitrarily small.
As a result, the cutoff tests obtain that 
$H_{0}$ is necessary and
$H_{1},\ldots,H_{n}$ are all impossible, that is,
they lack consonance with respect to the union operation.
A similar example shows that these tests also 
are not consonant with respect to intersection. 

The results of this section show that
several standard hypothesis tests cannot be made
logically consistent.
In the following section, we discuss how to obtain 
logically consistent hypothesis tests.

\section{Logically consistent agnostic hypothesis tests}
\label{section:region-tests}

A logically consistent agnostic hypothesis test satisfies
all the properties described in \Cref{section:properties}:
invertibility, monotonicity, union and intersection consonance,
and $\square \Theta$ is obtained.
\citet{Esteves2016} shows that 
every logically consistent agnostic testing scheme
can be obtained from a \emph{region estimator}
(a set of plausible values for $\theta$).
More precisely, for every logically consistent test,
there exists a region estimator, $S \subseteq \Theta$,
such that the test is of the form
\begin{align*}
 \mbox{Choose credal modality} \left\{
 \begin{array}{ll}
  \square H & \mbox{if} \  S \subseteq H \\
  \lnot \Diamond  H& \mbox{if} \  S \subseteq \bH   \\
  \nabla H  &  \mbox{if} \ S \cap H \neq \varnothing \mbox{ and }
  S \cap \bH \neq \varnothing
 \end{array}
 \right.
\end{align*}
The tests that are of the above type are called
\textit{region-based tests}.
Figure \ref{fig::region} illustrates such tests.		
\begin{figure}
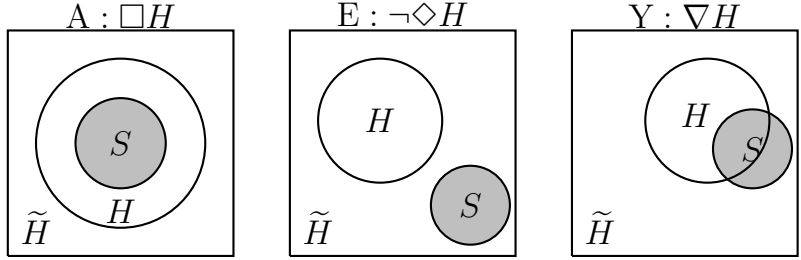

 \center
 \gfbstfig
 \mbox{} \vspace{-2mm} \mbox{} \\ 
 \caption{A test for hypothesis $H$ based on 
          the region estimator for $\theta$, $S$.}
 \label{fig::region}
\end{figure}
No matter what region estimator, $S$, is used,
the test based on $S$ is logically consistent \citep{Esteves2016}.
In order for the test to have additional statistical properties,
$S$ might be built using the Bayesian or the frequentist frameworks.
In the following,  we describe 
a particular type of region-based test:
the \emph{Generalized Full Bayesian Significance Test} (GFBST).
This test is an extension of the Full Bayesian Significance Test \citep{Pereira1999} (FBST) to agnostic hypothesis tests.

\subsection*{FBST and GFBST}
Credal modalities that are obtained from a possibilistic 
measure \citep{Dubois1982, Dubois2012} are defined as 
\textit{possibilistic} (or alethic) modal operators
($\square$, $\Diamond$, $\nabla$ and $\Delta$). 
Note that these symbols are different from
the ones used for the probability modal operators.
Hereafter, we study the possibilistic modal operators that
are obtained from the FBST and the GFBST.

The FBST is based on
the \emph{epistemic value} of 
the hypothesis of interest given 
the observed data, $\ev(H\g x)$.
The epistemic value is a transformation of
a probability into a possibility measure 
\citep{Borges2007, Stern2014}.
Specifically, $\ev(H\g x)$ is obtained through 
the steps:
\begin{enumerate}
 \item Define a surprise function,
 $s(\theta \g x) = \frac{p(\theta \g x)}{r(\theta)}$,
 where $r(\theta)$ is a \emph{reference} density over $\Theta$  \citep{stern2011}. The \textit{reference density} can be interpreted as 
an invariant measure under a relevant transformation group. For example,
the reference density may be obtained from the information geometry according to a metric on the parameter space \citep{Stern2011b}.
 
 \item Define the tangent set to $H$, $T(H)$, as
 $T(H)$, where, 
 \begin{align*}
  T(H) = \{\theta_1 \in \Theta \g  \forall \theta_{0} \in H, s(\theta_1 \g x) > s(\theta_0 \g x)\}\ 
 \end{align*}
 Note that $T(H) \cap H = \emptyset$, that is, the tangent set to an hypothesis and the hypothesis are disjoint. 
This aspect of the $e$-value's definition is
motivated by the legal principle known as  
\textit{in dubio pro reo} (most favorable interpretation, benefit of the doubt, or presumption of innocence) \citep{Stern2003}.
Intuitively, only parameter points that are 
more probable than every point of 
a given hypothesis can be admitted as 
witnesses against this hypothesis.  
    
 \item Obtain $\ev(H \g x) = 1-p\left(T(H) \g x \right)$.
 
 Alernatively, \Cref{lemma:sup-ev} in the appendix shows how
 $\ev(H \g x)$ can be obtained as a function of
 $\ev(\{\theta_0\} \g x)$,
 for each $\theta_0 \in H$.
\end{enumerate}

The FBST chooses $\square H$ 
when $\ev(H \g x) > c$, and
$\lnot \Diamond H$ 
when $\ev(H \g x) \leq c$, 
where $c \in (0,1)$ is a cutoff defined by 
the practitioner.
This procedure is similar to region-based tests.
Indeed, if $S=\{\theta_0 \in \Theta: \ev(\{\theta_0\} \g x) > c\}$, then
\Cref{thm:ev-region} shows that
the FBST has the form
\begin{align*}
 \mbox{Choose credal modality}  
 \left\{\begin{array}{ll}
  \lnot \Diamond H	& \mbox{, if} \ S \subseteq \bH \\
  \square H			&  \mbox{, otherwise.}			\\
 \end{array}
 \right.
\end{align*}

The GFBST extends the FBST into a region-based
agnostic hypothesis test. Let $S$ be such as in
the previous paragraph.
The GFBST has the form:
\begin{align*}
\mbox{Choose credal modality}  
 \left\{\begin{array}{ll}
  \lnot \Diamond  H& \mbox{, if} \  S \subseteq \bH \ (S \cap H = \emptyset) \\
  \square H & \mbox{, if} \  S \subseteq H \ (S \cap \bH = \emptyset) \\
  \nabla H  &  \mbox{, if} \ S \cap H \neq \varnothing \mbox{ and }
	S \cap \bH \neq \varnothing.
 \end{array}
 \right.
\end{align*}
While in the FBST, 
$S \cap H \neq \emptyset$ implies $\square H$,
in the GFBST this case can lead to 
either $\square H$ or $\nabla H$. Since this difference makes the GFBST a region-based test,
it is logically consistent. By applying \Cref{thm:ev-region}, the GFBST can also 
be written as a function of 
$\ev(H \g x)$ and $\ev(\bH \g x)$:
\begin{align}
 \label{eqn:gfbst-test}
 \mbox{Choose credal modality} 
 \left\{\begin{array}{ll}
  \lnot \Diamond  H 
  & \mbox{, if} \ \ev(H\g x) \leq c \\
  \square H
  & \mbox{, if} \ \ev(\bH\g x) \leq c \\
  \nabla H
  & \mbox{, otherwise.}
 \end{array}
 \right.
\end{align}
 
\textbf{Remark}: 
In previous references,
the tangent set to $H$ was defined as 
$T^{*}(H) = \{\theta_1 \in \Theta: s(\theta_1 \g x) > \sup_{\theta_0 \in H}s(\theta_0 \g x)\}$. 
Under regularity conditions that are 
usually found in statistical models
\citep[ex.3.23]{Izbicki2015},  
$p(T^*(H) \g x) = p(T(H) \g x)$.
Hence, in these cases, 
both definitions lead to the same results.
However, only $T(H)$ allows the 
assumption-free characterization 
in \Cref{eqn:gfbst-test}.  

The GFBST can be made more interpretable
using some properties of $\ev$.
Note that, for every $H$,
either $T(H) = \emptyset$ 
or $T(\bH) = \emptyset$.
Therefore, either $\ev(H \g x)=1$ or $\ev(\bH \g x)=1$.
When $\ev(H \g x)=1$ and $\ev(\bH \g x)=1$,
no matter what the value of $c$ is,
$\nabla H$ and $\nabla \bH$ are obtained.
When $\ev(H \g x)=1$ and $\ev(\bH \g x) < 1$,
either $\square H$ and $\lnot \Diamond \bH$ are obtained, or $\nabla H$ and $\nabla \bH$ are obtained. Similarly, if $\ev(\bH \g x)=1$ and $\ev(H \g x) < 1$, then either $\square \bH$ and $\lnot \Diamond H$ are obtained or $\nabla H$ and $\nabla \bH$ are obtained.
\Cref{sec::stackedCutoff} illustrates 
this behavior when 
$\ev(H \g x)=1$ and $\ev(\bH\g x) < 1$.
For every $c \in (0,1)$,
$\lnot \Diamond H$ is not obtained.
This behavior is different from the one in \Cref{fig:ripley}.

\begin{figure}
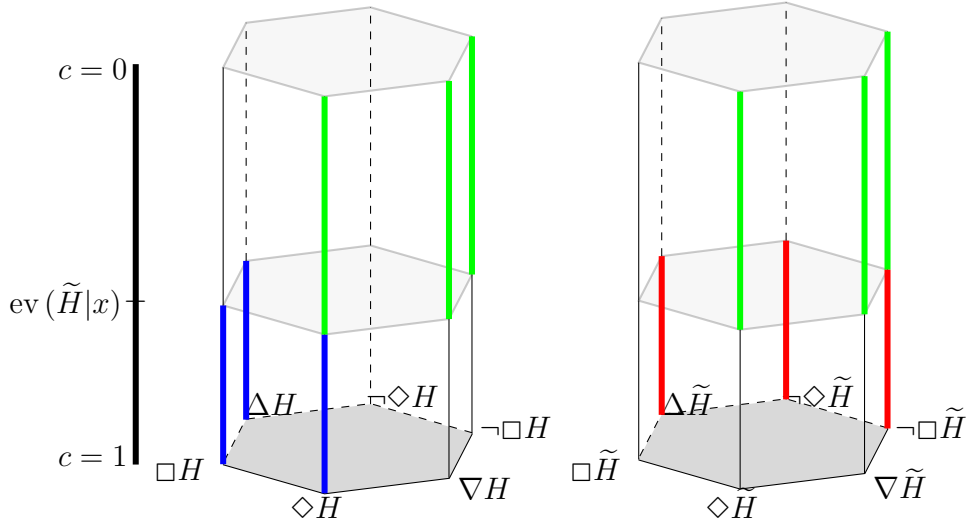

 \centering
 \gfbst
 \caption{Example of the behavior of the GFBST for a fixed sample $x$ as a function of the cutoff, $c$, when the supremum of the surprise function
is obtained only in $H$.
The prisms on the left and right side represent
the results obtained by the GFBST when applied, respectively, to $H$ and $\bH$.
The hexagon obtained from each horizontal cut represents the GFBST test that uses the corresponding cutoff.
The colored edges represent the credal modalities that are obtained from the GFBST as a function of the cutoff value.
 }
 \label{sec::stackedCutoff}
\end{figure}

\section{Hybrid hexagons: Relations between probabilistic and possibilistic modal operators}        
\label{section:hybrid}

Although the posterior probability test and the GFBST have different logical properties,   
It follows from the characterization of the GFBST in \Cref{section:region-tests}, 
that $\square H$ can be made 
a more stringent modality than $\dib H$.
Indeed, if $c < 0.5$ in \Cref{eqn:gfbst-test}, then
\Cref{thm:ev-prob} in the Appendix
proves that
\begin{align}
 \label{eq:gfbst-prob}
 \square H \Rightarrow p(H \g x) \geq 1-c \Rightarrow \Diamond H
\end{align}
Therefore, in \Cref{eqn:cutoff-test},
the cutoffs $c_1=1-c$ and $c_2=c$ yield
\begin{align*}
 \square H \Rightarrow \sqb H \Rightarrow \ \dib H \Rightarrow \Diamond H
\end{align*}
These relationships are summarized in the nested
hexagon of Figure \ref{figure:hexa-agnosticEncaixado}.
The above implications show that it is possible
to combine the two types of credal modality into 
a single hexagon of oppositions.
Since the hexagon of oppositions has two degrees of freedom
and the implication $\square H \Rightarrow \dib H$ holds,
it is possible to use $\square H$ as the necessity modality
and $\dib H$ as the possibility modality in a new 
hexagon of oppositions.
Similarly, since $\sqb H \Rightarrow \Diamond H$ holds,
one can use $\sqb H$ as the necessity modality and
$\Diamond H$ as the possibility modality in another hexagon of oppositions.
Similar hybrid hexagons have appeared in
\citet{Carnielli2008, Demey2014}.

\begin{figure}
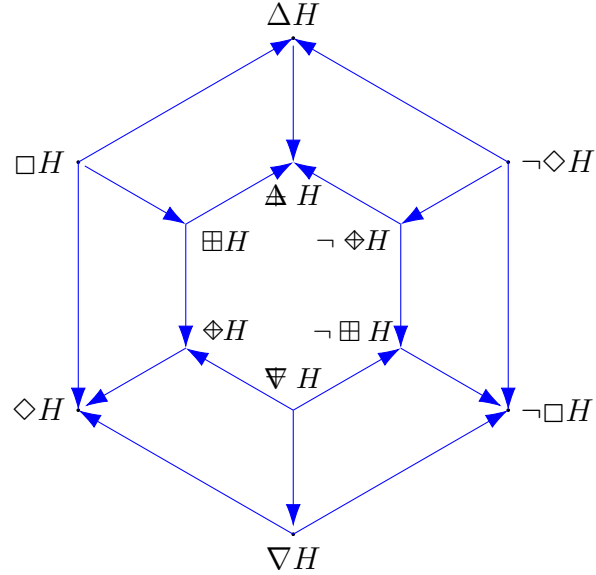

 \centering
 \hexagonoEncaixado
 \caption{Nested hexagon of opposition that shows the implication relations between  the probabilistic and possibilistic
 modal operators.}
 \label{figure:hexa-agnosticEncaixado}
\end{figure}

\begin{figure}
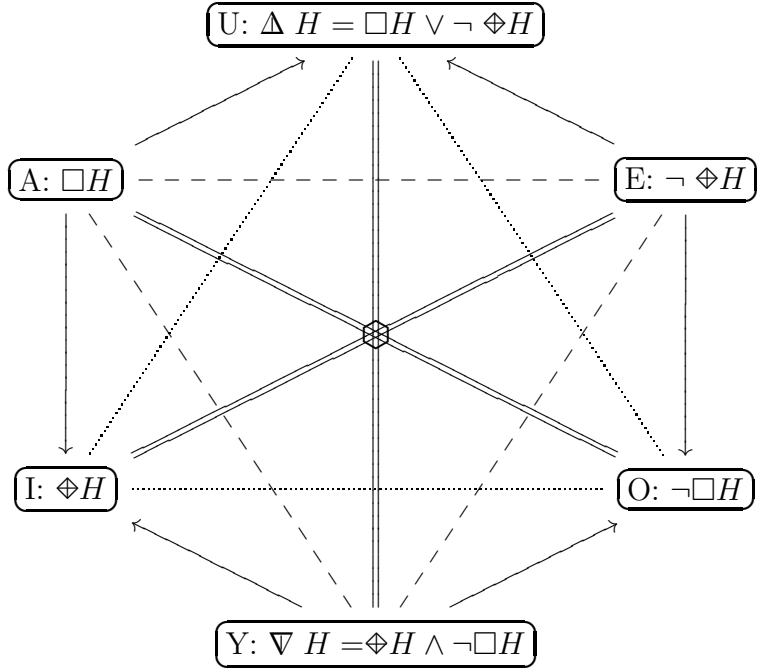

 \begin{align*}
  \hexagonoh
 \end{align*}
  \caption{Hybrid hexagon of oppositions that
  is obtained from combining probabilistic and possibilistic  
	 modalities.}
 \label{fig:hybrid-hexa}
\end{figure}

\Cref{fig:hybrid-hexa} illustrates 
the hybrid hexagon obtained using
$\square H$ and $\dib H$.
This hexagon summarizes several relations
between  $\square H$ and $\dib H$ that 
can be proved using the fact that
$\square H \Rightarrow \dib H$.
In particular, note that, if $p(H) = 0$,
then it follows from \Cref{eq:bayes} that $p(H \g x)=0$.
Therefore, it follows from \Cref{eqn:cutoff-test} that 
$\lnot \dib H$ holds.
Using the contrariety relation in \Cref{fig:hybrid-hexa},
one can conclude that $\square H$ does not hold.
In summary, if a hypothesis, $H$, is such that $p(H)=0$, then
$H$ is not accepted by the GFBST.

The previous conclusion is compatible with a common idea in Statistics: 
if $p(H) = 0$, then one cannot accept $H$; 
one can only fail to reject $H$ \citep{Gigerenzer2004, Inman1994}. 
Indeed, if $p(H) = 0$, then
$\lnot \square H$ is obtained, that is, 
the GFBST cannot accept $H$. However, 
even though $p(H) = 0$, it is possible to 
obtain $\Diamond H$, that is, the GFBST can fail to reject $H$.
However, when $p(H) =0 $, then
$p(H \g x)=0$ and 
$\lnot \dib H$ is obtained,
that is, the posterior probability test always rejects $H$.

The above result might be surprising,
given the importance of null probability hypotheses 
in the theoretical and empirical sciences \citep{Stern2015}. 
For example, if $\Theta = \mathbb{R}^{d}$, then
it is common to find hypotheses in the form of linear equations, 
$H: A\theta = b$, where matrix $A$ is $h \times d$.
This case includes, for example, the simple null hypothesis $H: \theta = c$. 
In case the distribution of $\theta$ is continuous, then
$p(H) = 0$ and, for every $x$, $p(H \g x) = 0$.
Therefore, while the probability cutoff test always obtains $\lnot \dib H$ for precise hypotheses,
the GFBST can either 
obtain $\nabla H$ or $\lnot \Diamond H$.
That is, while the credal modality obtained from the probability cutoff test for precise hypotheses is known beforehand, 
the GFBST allows one to use data to 
revisit their beliefs
regarding these hypotheses.

\section{Conclusions and Final Remarks}

 We show how the hexagon of oppositions can be a useful tool to explain consistency properties of statistical hypothesis tests. Indeed, geometric solids composed of hexagons of oppositions illustrate the conditions for a statistical hypothesis test to be logically consistent. Also, prisms composed of hexagons of oppositions explain the definitions of hypothesis tests such as the probability cutoff test and the GFBST. A hybrid hexagon of oppositions summarizes the logical relations between these tests. 
 In summary, the hexagon of opposition can be used as a powerful form of diagrammatic representation that is helpful in organizing, displaying and explaining complex logical properties of multiple statistical tests of hypothesis and their intricate inter-relationships.      
 
{\bf Acknowledgments:} 
 The authors are grateful for the support of IME-USP, the Institute of Mathematics and Statistics of the University of S\~{a}o Paulo, and the Department of Statistics of UFSCar - The Federal University of S\~{a}o Carlos. 
 Julio M. Stern is grateful to 
FAPESP - the State of S\~{a}o Paulo Research Foundation (grants CEPID 2013/07375-0 and CEPID 2014/50279-4); and CNPq - the Brazilian National Counsel of Technological and Scientific Development (grants PQ 301206/2011-2 and PQ 301892/2015-6). 
 Rafael Izbicki is grateful for the support of FAPESP (grants 2014/25302-2 and 2017/03363-8). 
 Finally, the authors are grateful for advice and comments received from anonymous referees, and from participants of the 5th World Congress on the Square of Opposition, held on November 11-15, 2016, at Hanga Roa, Rapa Nui, Chile, having as main organizer Jean-Yves B\'{e}ziau; and the XVIII Brazilian Logic Conference, held on May 8-12, 2017, at Piren\'{o}polis, Goi\'{a}s, Brazil, co-chaired by Marcelo Coniglio, Itala D'Ottaviano and Wagner Sanz.

\bibliographystyle{apalike} 
\bibliography{hexa}

\vspace{-0.5cm}

\section*{Proofs}

\begin{lemma}
 \label{lemma:logical}
 If $\Delta H_1$ and $\Delta H_2$ are obtained
 in a logically consistent hypothesis test, then
 $\square(H_1 \uparrow H_2)$ is obtained
 if and only if 
 $\square H_1 \uparrow \square H_2$ is obtained.
\end{lemma}

\begin{proof}
 If $\square(H_1 \uparrow H_2)$ is obtained,
 then $\square(\Theta - (H_1 \cap H_2))$ is obtained.
 It follows from invertibility that
 $\neg \Diamond(H_1 \cap H_2)$ is obtained.
 Therefore, it follows from an implication relation in
 \Cref{figure:hexa-agnostic} that
 $\neg \square(H_1 \cap H_2)$ is obtained.
 Hence, using intersection consonance, 
 either $\neg \square H_1$ or 
 $\neg \square H_2$ is obtained.
 Conclude from a contradiction relation in
 \Cref{figure:hexa-agnostic} that
 either $\square H_1$ is not obtained or 
 $\square H_2$ is not obtained, that is,
 $\square H_1 \uparrow \square H_2$ is obtained.
 
 If $\square H_1 \uparrow \square H_2$ is obtained, then
 either $\square H_1$ is not obtained or 
 $\square H_2$ is not obtained.
 Hence, using a contradiction relation in
 \Cref{figure:hexa-agnostic},
 either $\neg \square H_1$ or 
 $\neg \square H_2$ is obtained.
 Since by hypothesis $\Delta H_1$
 and $\Delta H_2$ are obtained,
 it follows from \Cref{table:modalities} that
 either $\neg \Diamond H_1$ or
 $\neg \Diamond H_2$ is obtained.
 Therefore, using monotonicity,
 $\neg \Diamond (H_1 \cap H_2)$ is obtained.
 It follows from invertibility that
 $\square (\Theta - (H_1 \cap H_2))$ is obtained,
 that is, $\square (H_1 \uparrow H_2)$ is obtained.
\end{proof}

\begin{lemma}
 \label{lemma:sup-ev}
 $\ev(H \g x) = \sup_{\theta_0 \in H}\ev(\{\theta_0\} \g x)$
\end{lemma}

\begin{proof}
 \begin{align*}
  \ev(H \g x) 
  &= 1-p(\{\theta_1 \in \Theta \g  \forall \theta_{0} \in H, s(\theta_1 \g x) > s(\theta_0 \g x)\}\ \g x) \\
  &= 1-p(\cap_{\theta_{0} \in H} \{\theta_1 \in \Theta \g s(\theta_1 \g x) > s(\theta_0 \g x)\}\ \g x) \\
  &= 1-\inf_{\theta_{0} \in H}p(\{\theta_1 \in \Theta \g s(\theta_1 \g x) > s(\theta_0 \g x)\}\ \g x) \\
  &= 1-\inf_{\theta_{0} \in H}p(T(\{\theta_0\})\g x) = \sup_{\theta_{0} \in H}ev(\{\theta_0\} \g x)\\
 \end{align*}
\end{proof}

\begin{theorem}
 \label{thm:ev-region}
 Let $S = \{\theta_{0} \in \Theta: \ev(\{\theta_{0}\} \g x) > c\}$. For every $H \subseteq \Theta$, $\ev(H \g x) \leq c$ if and only if $H \cap S = \emptyset$.
\end{theorem}

\begin{proof}
 Note that $\ev(H \g x) \leq c$ if and only if
 $\sup_{\theta_0 \in H}\ev(\{\theta_0\} \g x) \leq c$ (\Cref{lemma:sup-ev}).
 That is, $\ev(H \g x) \leq c$ if and only if
 $S \cap H = \emptyset$.
\end{proof}

\begin{theorem}
 \label{thm:ev-prob}
 If $c < 0.5$ in 
 \Cref{eqn:gfbst-test}, then
 \begin{align*}
  \square H \Rightarrow 
  p(H \g x) \geq 1-c \Rightarrow 
  p(H \g x) > c \Rightarrow 
  \Diamond H
 \end{align*}
\end{theorem}

\begin{proof}
 For every $\Theta_{H} \subset \Theta$,
 \begin{align}
  1-\ev(\bH \g x) 
  &=p(\theta \in T(\bH ) \g x)  \nonumber \\
  &= p(\theta \in \{\theta_1 \in \Theta \g  \forall \theta_{0} \in \bH , s(\theta_1 \g x) > s(\theta_0 \g x)\}\ \g x) \nonumber \\
  &\leq p(\theta \in \Theta_{H} \g x)
  = p(H \g x), \label{eq:ev-inequality-1}
 \end{align}
 where the last inequality follows from the fact that
 \begin{align*}
  \{\theta_1 \in \Theta \g  \forall \theta_{0} \in \bH , s(\theta_1 \g x) > s(\theta_0 \g x)\}
  &\subseteq \Theta-\Theta_{\bH } = \Theta_{H}.
 \end{align*}
 Similarly, one can obtain
 \begin{align}
  \label{eq:ev-inequality-2}
  1-\ev(H \g x) \leq p(\bH  \g x)
 \end{align}
 
 Assume that $\square H$ is
 obtained. It follows from 
 \Cref{eqn:gfbst-test} that 
 $\ev(\bH \g x) \leq c$, that is,
 $1-\ev(\bH \g x) \geq 1-c$.
 Conclude from \Cref{eq:ev-inequality-1}
 that $p(H\g x) \geq 1-c$.
 Then, since $c<0.5$,
it holds that $p(H | x) > c$.
 Therefore, $p(\bH \g x) < 1-c$.
 Conclude from \Cref{eq:ev-inequality-2}
 that $1-\ev(H \g x) < 1-c$, that is,
 $\ev(H \g x) > c$.
 It follows from \Cref{eqn:gfbst-test} that
 $\Diamond H$ is obtained.
 \vspace{-1cm}
\end{proof}

\end{document}